\documentclass[11pt]{amsart}

\usepackage[T1]{fontenc}
\usepackage[utf8]{inputenc}
\usepackage{lmodern}

\usepackage[american]{babel}
\usepackage[babel, final]{microtype}
\usepackage{amssymb}
\usepackage{etoolbox}

\usepackage{lineno}

\usepackage{comment}
\usepackage[pdftex, dvipsnames]{xcolor}
\usepackage[pdftex, final]{graphicx}
\usepackage{caption}
\usepackage{subcaption}
\usepackage{pinlabel}
\usepackage[pdftex,%
letterpaper,%
includehead,%
includefoot,%
nomarginpar,%
lmargin=1in,%
rmargin=1in,%
tmargin=1in,%
bmargin=1in,%
]{geometry}
\usepackage[pdftex,%
final,%
colorlinks=false,%
bookmarks=true,%
bookmarksdepth=3,%
bookmarksnumbered=true,%
bookmarksopen=true,%
bookmarksopenlevel=2,%
]{hyperref}
\hypersetup{
pdftitle={Using knot Floer theory to detect primality},
pdfauthor={Samantha Allen, Charles Livingston, Misha Temkin, C.-M. Michael
Wong},
pdfsubject={Knot Floer theory},
pdfkeywords={Knot Floer theory, prime knots}
}

\newlength\chartspace
\usepackage{tikz}
\usetikzlibrary{arrows.meta}
\usetikzlibrary{positioning}
\tikzset{arrow/.style={
-{Latex[length=1.5mm,    %! properties of the...
width=1.5mm]},  %! ...arrow tip
line width      =0.7pt,    %! width of an arrow
shorten        >=2pt,      %! a room between arrowhead and a node
shorten        <=2pt,      %! a room between arrowtail and a node
}}

\tikzset{dot/.style={
circle,
fill        =black,
inner sep   =0pt,
outer sep   =0pt,
minimum size=2*\dotsize cm,
}}

\tikzset{lbl/.style args={#1/#2}{
#1  =\labeloffset of #2,
font=\fontsize{0.3 cm}{0}\selectfont,  %! fontsize of a label. don't touch zero
}}

\makeatletter
\g@addto@macro\@floatboxreset\centering
\makeatother

\allowdisplaybreaks[4]

\AtBeginDocument{%
\def\MR#1{}
}

\patchcmd{\thebibliography}{\advance\leftmargin\labelsep}
{\labelsep=0.8em \advance\leftmargin\labelsep}{}{}

\newlength\openbractt
\newtheorem{theorem}{Theorem}
\newtheorem{lemma}[theorem]{Lemma}
\newtheorem{corollary}[theorem]{Corollary}
\newtheorem{proposition}[theorem]{Proposition}

\theoremstyle{definition}

\newtheorem{remark}[theorem]{Remark}

\makeatletter
\newtheorem*{rep@theorem}{\rep@title}
\newcommand{\newreptheorem}[2]{%
\newenvironment{rep#1}[1]{%
\def\rep@title{#2 \ref{##1}}%
\begin{rep@theorem}}%
{\end{rep@theorem}}}
\makeatother

\newreptheorem{theorem}{Theorem}

\renewcommand*{\leq}{\leqslant}
\renewcommand*{\le}{\leqslant}

\newcommand*{\co}{\colon\thinspace}

\newcommand*{\Z}{\mathbb{Z}}

\newcommand*{\F}{\mathbb{F}}

\newcommand*{\cs}{\mathbin{\sharp}}

\DeclareMathOperator{\CFK}{CFK}
\DeclareMathOperator{\HFK}{HFK}
\newcommand*{\CFKh}{\widehat{\CFK}}

\newcommand*{\HFKh}{\widehat{\HFK}}

\newcommand*{\be}{b_{\mathrm{e}}}
\newcommand*{\bo}{b_{\mathrm{o}}}

\begin{document}
\title{Using knot Floer invariants to detect prime knots}

\author{Samantha Allen}
%\thanks{This work was supported by a grant from the National Science Foundation, NSF-DMS-1505586.  }
\address{Department of Mathematics and Computer Science,  Duquesne University, Pittsburgh, PA 15282}
\email{\href{mailto:allens6@duq.edu}{allens6@duq.edu}}
\urladdr{\url{https://samantha-allen.github.io/}}

\author{Charles Livingston}
%\thanks{This work was supported by a grant from the National Science Foundation, NSF-DMS-1505586.  }
\address{Department of Mathematics, Indiana University, Bloomington, IN 47405}
\email{\href{mailto:livingst@indiana.edu}{livingst@indiana.edu}}
\urladdr{\url{https://livingst.pages.iu.edu/}}

\author{Misha Temkin}
%\thanks{This work was supported by a grant from the National Science Foundation, NSF-DMS-1505586.  }
\address{Department of Mathematics, Dartmouth College, Hanover, NH 03755}
\email{\href{mailto:misha.temkin@dartmouth.edu}{misha.temkin@dartmouth.edu}}
\urladdr{\url{https://math.dartmouth.edu/~mt/}}

\author{C.-M. Michael Wong}
%\thanks{This work was supported by a grant from the National Science Foundation, NSF-DMS-1505586.  }
\address{Department of Mathematics and Statistics, University of Ottawa, Ottawa, ON K1N 6N5 }
\email{\href{mailto:Mike.Wong@uOttawa.ca}{Mike.Wong@uOttawa.ca}}
\urladdr{\url{https://mysite.science.uottawa.ca/cwong/}}

%\date{\today}

%%%%%%%ABSTRACT%%%%%%%%%%%%%%

\begin{abstract}  
We present knot primality tests that are built from knot Floer homology.  The most basic of these is a simply stated and    elementary consequence of Heegaard Floer theory:  if the knot Floer polynomial $\Omega_K(s,t)$ is irreducible, then $K$ is prime.   Improvements in this test yield a primality condition that is over 92\% effective in identifying prime knots of up to 15 crossings.   As another illustration of the strength of these tools, there are 1,315 non-hyperbolic prime knots  with crossing number 20 or less; the tests we develop   prove  the  primality of over 97\% of them.

The  filtered chain homotopy class  of $\CFKh (K)$ has a unique minimal-dimension representative
that is the  direct sum of a one-dimensional complex and
two-dimensional complexes, each of which can be assigned a  parity.  Let
$\delta (K)$, $\be (K)$, and $\bo (K)$  denote the   dimension of this minimal
representative  and the number of even and odd two-dimensional summands, respectively. For
a composite knot $K$, we observe that there is a non-trivial factorization $\delta (K) = d_1d_2$ satisfying
\[ (d_1 -1)(d_2-1) \le 4\min(\be(K), \bo(K)) .\]
This yields another knot primality test.  One corollary is a simple proof of Krcatovich's result that $L$-space knots are prime.

\end{abstract}

\maketitle

%%%%%%%Section%%%%%%%%%%%%%%

\section{Introduction}  The main tools for proving that a   knot $K \subset S^3$ is prime have been geometric, depending on hyperbolic geometry, incompressible surfaces, or geometric decompositions of the knot complement.  Background is provided in~\cite{MR3156509}; more recent discussions include~\cite{burton:LIPIcs:2020:12183, MR787823, Thistlethwaite2023}.   Here we will explore algebraic  knot primality tests based on Floer knot theory. For example, the most basic of these tests is a quick consequence of Heegaard Floer theory:  {\it If the two-variable knot Floer polynomial $\Omega_K(s,t)$ built from $\HFKh(K)$ is irreducible, then $K$ is prime.}

In the case that one has an alternating diagram for $K$,  Menasco~\cite{MR721450} gave simple necessary and sufficient conditions that determine if  $K$ is  prime.   More relevant to our work  is the fact that if a knot  is alternating, then the degree of its Alexander polynomial is twice its genus~\cite{MR99665, MR99664,MR1988285}.  As a consequence, there is a primality test for alternating knots $K$ in terms of the Alexander polynomial: {\it If  $K$ is alternating and $\Delta_K(t)$ is irreducible, then $K$ is prime.}  The primality tests we describe    generalize this to  all knots and are much stronger, even in the case of alternating knots.
%Even in the case of alternating knots, the tests we are develop are much stronger than the Alexander polynomial test; for instance, the Alexander polynomial test fails for roughly 2,000 of the alternating knots of 13 crossings or less; our tests resolve almost half of these examples.

We will be working with the  knot Floer   chain complex   $\CFKh(K)$ and its associated graded homology $\HFKh(K)$, first defined in~\cite{MR2065507}.  We will work exclusively with coefficients in the field with two elements, $\F_2$.  The computation of this complex is algorithmic~\cite{MR3771149,MR4002230}.   A program  written by Zolt\'an Szab\'o,  enhanced by
Jonathan Hanselman, and now installed in SnapPy~\cite{SnapPy} has been used to develop examples.   

In the next section we present  primality tests that are built using $\HFKh(K)$.   A complete survey of prime knots with  crossing number at most 15 and a large random sample of knots through 30 crossings indicate that these tests prove primality for over 92\% of prime knots. Surprisingly, the tests appear to be slightly stronger for non-alternating knots.
In Section~\ref{sec-bar} we consider the filtered chain complex $\CFKh(K)$, taking advantage of its differential as well as its homology.  The tests using this added structure are more specialized; so far they appear to be of  theoretical, rather than computational, interest.  For instance, this approach offers a simple proof of Krcatovich's theorem that $L$--space knots are prime~\cite{MR3404612}; see Corollary~\ref{cor-l-space}.

\subsection*{Acknowledgments} The authors thank Zolt\'an Szab\'o and Jonanthan Hanselman for their pointers regarding the use of Zolt\'an's knot Floer program.  The authors appreciate Nathan Dunfield's guidance regarding SnapPy.  Feedback from Peter Ozsv\'ath led us to further explore the case of alternating knots.  John Baldwin pointed out his work with Steven Sivek that showed that   $\HFKh(K)$ identifies the  knot $5_2$ and  the family of pretzel knots $\{P(-3,3,2n+1)\}$, among others; these examples  significantly enhance the strength of the primality tests presented here.

SA thanks the Wimmer Family Foundation for its support. CMMW acknowledges the support of the Natural Sciences and Engineering Research Council of Canada (NSERC), RGPIN-2023-05123.

%%%%%%%Section%%%%%%%%%%%%%%

\section{Knot Floer homology primality obstructions}
\label{sec-intro}

Let $\CFKh(K)$ denote the knot Floer homology complex~\cite{MR2065507} with coefficients in the field with two elements.  This is a finite-dimensional, graded, filtered chain complex that is well-defined up to filtered chain homotopy equivalence.  We let  $\CFKh(K, i)$ denote the subcomplex at filtration level $i \in \Z$.  The group $\HFKh_j(K, i)$ is the   homology of the quotient complex $\CFKh(K,i)/\CFKh(K,i-1)$ at grading  $j$.  We define the knot Floer polynomial by
\[ \Omega_K(s,t) = \sum \dim(\HFKh_j(K, i))s^j t^i .\]

A feature of knot Floer homology is that it is multiplicative:  $\CFKh(K \cs J) = \CFKh(K) \otimes \CFKh(J)$.  From this it follows that $\Omega_{K \cs J}(s,t) = \Omega_{K}(s,t) \Omega_{  J}(s,t) $.  (See~\cite{MR3591644} for a brief discussion of multiplicativity; we will present a few details in the appendix.)  Futhermore, Ozsv\'ath and Szab\'o proved in~\cite{MR2023281} that $\HFKh(K)$ detects the genus of a knot; in particular, $\Omega_K(s,t) = 1$ if and only if $K$ is the unknot.  Thus we have a basic primality test.

\begin{proposition} If $\Omega_K(s,t) $ is irreducible, then $K$ is prime. \end{proposition}

\noindent Note there are algorithms for determining whether or not a   multivariable polynomial over the integers  is irreducible. \smallskip

This result  provides a primality test that has been effective in roughly 78\% of our experiments, which have included prime knots with up to 30 crossings.  There are a few refinements that make it  more effective, as we now describe.

According to~\cite{MR2065507}, the   polynomial $\Omega_K(s,t)  = \sum c_{i,j}s^jt^i $ satisfies the symmetry condition  $c_{i,j} = c_{-i, j-2i}$ for all $i$ and $j$.   We say that a factorization of a two-variable  Laurent polynomial is  {\it symmetric} if each factor satisfies this symmetry condition.  If no such nontrivial factorization exists, the polynomial is called {\it symmetrically irreducible}.

\begin{theorem} \label{thm-theorem2} If $\Omega_K(s,t) $ is symmetrically irreducible, then $K$ is prime. \end{theorem}

It is known that knot Floer homology detects the torus knots $T(2,3)$ and $T(2,5)$ as well as the figure eight knot, $4_1$ (see~\cite{farber2023,MR2450204}).  More recently, it has been shown that it also detects the knots $5_2$ and the Whitehead double $\text{Wh}^+(T(2,3), 2)$     (see~\cite{2022arXiv220803307B}).   Another result of ~\cite{2022arXiv220803307B} is that $\HFKh(K)$ determines whether a knot is one of  $\text{Wh}^-(T(2,3),2)$ and  $15n_{43522}$, which have identical $\HFKh(K)$.  Lastly,~\cite{2022arXiv220803307B} proves that $\HFKh(K)$ detects membership in the set of pretzel knots $\{P(-3,3,2n+1)\}_{n \in \Z}$; more precisely, a knot  $K$ is in this family if and only if $\Omega_K(s,t) =   2st    + 5+	2s^{-1}t^{-1}$.

  Here is a    list  of  all nontrivial knots that are known to be determined by $\CFKh(K)$, along with their knot Floer polynomials  $\Omega_K(s,t)$.  In this list  we do not include mirror images, $m(K)$, since  $\Omega_{m(K)}(s,t) = \Omega_K(s^{-1},t^{-1})$. \smallskip
  
  \begin{itemize}
  \item {\makebox[\chartspace][l]{$T(2,3) \co$}} $t+ s^{-1} + s^{-2}t^{-1}$.\smallskip
  \item {\makebox[\chartspace][l]{$T(2,5) \co$}} $t^2 + s^{-1}t + s^{-2} + s^{-3}t^{-1} + s^{-4}t^{-2}$.\smallskip
  \item {\makebox[\chartspace][l]{$4_1 \co$}} $st + 3 + s^{-1}t^{-1}$.\smallskip
  \item {\makebox[\chartspace][l]{$5_2 \co$}} $2s^2 t + 3 s + 2t^{-1}$.\smallskip  
     \item {\makebox[\chartspace][l]{$Wh^+(T(2,3),2) \co$}}  $2s^{-1}t + (1+ 4s^{-2})+2s^{-3}t^{-1}$.\smallskip
  \item {\makebox[\chartspace][l]{$15n_{43522} \text{\ and\ } Wh^-(T(2,3),2)\co$}}  $2t+ (1+4s^{-1})+2s^{-2}t^{-1}$.\smallskip
  \item {\makebox[\chartspace][l]{$P(-3,3,2n+1)\co$}} $2st + 5 + 2s^{-1}t^{-1}$.

\end{itemize}
  \smallskip
  
\begin{theorem} \label{thm-theorem3} If $K$ is composite and  $\Omega_K(s,t) $ factors into exactly two symmetrically irreducible factors, one of which  is in the list above,  then $K$ has a factor that is the corresponding knot in the list. \end{theorem}

In practice, strong tools are available  that often can be used to prove that a knot  $K$  does not have any  of  these knots as factors;   multiplicative invariants such as the Jones polynomial are sufficient, though among prime knots with at most 15 crossings, we did find a few examples in which the HOMFLY polynomial (see~\cite{MR1472978}) was required.

For the set of knots $\{P(-3,3,2n+1)\}$, an exercise with the recursive property of the Jones polynomial $J_K(t)$ shows that  for all $n$,
     \[  J_{P(-3,3,2n+1)}(t) = t^2  J_{P(-3,3,2n-1)}(t) - t^2+ 1 .\]
     From this, one can show that the breadth of the Jones polynomial goes to infinite as $n$ becomes large, and thus checking whether the Jones polynomial of a given knot $K$ is divisible by the Jones polynomial of a knot in the set $\{P(-3,3,2n+1)\}$ is a   finite process  which is quickly done  in practice.

\vskip.1in

\noindent{\bf Summary of computations.}

In a complete survey  of prime knots of 15 or fewer crossings (of which there are 313,230),  Theorem~\ref{thm-theorem2} proved the primality of just under  $80\%$ of them.  Applying Theorem~\ref{thm-theorem3} increased that percentage to over 92\%.    For a random sampling of prime knots with 25 to 30 crossings  the success rate  was  above 95\%. Overall, it appears that the tests are slightly more effective for non-alternating knots.  % Since non-alternating knots dominate for higher crossing numbers, we expect that our tests will generally be over 92\% effective.

In the enumeration of prime knots, which has now been carried out through 20 crossings independently by Burton and Thistlethwaite, one of the most challenging aspects has been in proving the primality of those that are not hyperbolic.  Papers on this topic include~\cite{burton:LIPIcs:2020:12183, Thistlethwaite2023, MR787823}.  In that enumeration, of the more than two billion such knots, there are 1,315 that are not hyperbolic, including nine torus knots.  Applying Theorem~\ref{thm-theorem3} proves the primality of 1,285 of them, yielding a success rate of almost 98\%.

%%%%%%%Section%%%%%%%%%%%%%%
\section{Bar-complexes and further primality obstructions}\label{sec-bar}

An elementary homological result states that any finite-dimensional, graded, filtered chain complex is filtered chain homotopy equivalent to an essentially unique  {\it bar-complex}.  Here, by {\it bar-complex} we mean a complex along with a direct sum decomposition into one-dimensional complexes $(F_i)$    and two-dimensional complexes $(T_i \to B_i)$ which are acyclic as chain complexes and for which the filtration level of $B_i$ is strictly less than that of $T_i$.
In the language that is now common in persistent homology literature,  intervals bounded by the filtration levels of $B_i$ and $T_i$  are called {\it bars}. To simplify the discussion, we will   modify  this notation  and call the two-dimensional summands {\it bars}.
A bar is called {\it even} or {\it odd} depending on whether the grading of $B_i$ is even or odd, respectively.

Proofs of the existence and uniqueness of bar-complexes appear in such sources as~\cite{MR1310596, MR4249570}; going further back, see also~\cite{MR0001927}.
In forthcoming work, the authors will present a more general use of persistent
homology to study the doubly filtered complex $\CFKh(K)$;   see~\cite{allen2022upsilon} for background. In  knot
Floer theory literature, the bar-complex appears in the context of a reduced complex with a
vertically simplified basis; see, for example,~\cite[Definition
11.23]{MR3827056}. Here is a brief summary, in spirit close to that in~\cite{MR1310596}.

Choose an element $x$ of least filtration level for which $\partial x = y \ne 0$.  One can choose a filtered basis so  that $x$ and $y$ are  elements of that basis.  By adding $x$ to other filtered basis elements we can ensure that $x$ is the only  basis element whose boundary has a nontrivial $y$--coefficient.  The fact that $\partial^2 = 0$ implies that the pair $\{x, y\}$ generates a summand.  We now have that  $x$ and $y$ split off a two-dimensional summand.  If the filtration levels are the same, removing the summand is a filtered chain homotopy equivalence.  If the filtration levels are different, we have split off a  bar.

The number of bars which pair elements of specified filtrations levels can be expressed in terms of homological information derived from the maps $H_*(\CFKh(K, i)) \to H_*(\CFKh(K,j))$.  This implies the uniqueness of the bar-complex.

\subsection{Bar-complexes for $\CFKh(K)$ and connected sums}

Tensor products of bar-complexes arise naturally when considering connected sums of knots.  Thus, we will be taking advantage of the fact that the tensor product of two bar-complexes is  a bar-complex.  More precisely, we have the   following lemma, the proof of which is an easy computation that we leave to the reader.

\begin{lemma}\label{lem-bartensor}  The tensor product  of bars, $(T_1 \to B_1) \otimes (T_2 \to  B_2)$, is isomorphic to the direct sum of bars, $(X_1 \to Y_1) \oplus (X_2 \to Y_2)$,   where the grading and
filtration levels of $Y_1$ are the sums of those for $B_1$ and $B_2$ and the grading and filtration levels of
$X_2$ are the sums of those for $T_1$ and $T_2$. In particular, one of the summands is even and one is odd.

\end{lemma}

As an unfiltered complex, the homology of $\CFKh(K)$ is one-dimensional, supported in grading 0.  (The corresponding filtration level  is the $\tau$--invariant defined by Ozsv\'ath--Szab\'o in~\cite{MR2026543}.)   For a knot $K$,  let $\widehat{C}(K)$ denote the  bar-complex representative of $\CFKh(K)$.  We have the following definitions.

\begin{itemize}

\item $\delta(K) = \dim   \widehat{C}(K)$.  One can see that $\delta(K)$ is the  total dimension of $\HFKh(K)$.

\item $b_e(K)$ and $b_o(K)$ are  the number of even  and odd bars, respectively,  in $\widehat{C}(K)$.  Note that $\delta(K) = 1 +2\left(b_e(K) + b_o(K)\right)$.

\end{itemize}

An easy counting argument proves the following result.

\begin{theorem}\label{thm-counts}
For any two knots $K_1$ and $K_2$,   $\delta(K_1 \cs K_2 ) = \delta(K_1)\delta(K_2)$.  Furthermore,
\begin{align*}
\be(K_1 \cs K_2) & = \be(K_1) + \be(K_2) + \frac{1}{4} (\delta(K_1) -1)(\delta(K_2) -1);\\
\bo(K_1 \cs K_2) & = \bo(K_1) + \bo(K_2) + \frac{1}{4} (\delta(K_1) -1)(\delta(K_2) -1).
\end{align*}
\end{theorem}

Here is an evident  corollary.

\begin{corollary}\label{thm-evenodd} If $K$ is composite, then there is a
nontrivial factorization $\delta(K) = d_1 d_2$  for which
\[
(d_1-1)(d_2-1) \leq 4 \min (\be(K), \bo(K)).
\]
\end{corollary}

\subsection{Applications of Corollary~\ref{thm-evenodd}}

\noindent{\bf $L$--space knots are prime.}
An immediate consequence is a theorem of Krcatovich~\cite[Theorem~1.2]{MR3404612}; subsequent proofs   have appeared in~\cite{Baldwin2018, Boyer2023JSJ}.

\begin{corollary}[Krcatovich]
\label{cor-l-space}
$L$-space knots are prime.
\end{corollary}

\begin{proof}
It follows from the proof of
\cite[Theorem~1.2]{MR2168576} that if $K$ is an $L$-space knot,  then $b_o(K) = 0$.  Here are a few details.  The theorem states that there is a  filtered basis for $\CFKh(K)$ consisting of single generators at a strictly increasing set of filtration levels $f_0, f_1, \ldots , f_{2k}$, for some $k$.  The gradings of these generators alternate in parity, are strictly increasing, and the top level generator is at grading 0.  The concluding paragraph of the proof of    \cite[Theorem~1.2]{MR2168576} shows that the homology group $H(\CFKh(K, f_i)) = 0$ for $i$ odd.  Thus, the bar-complex consists of a single generator summand at filtration level $f_{2k}$  and   bars $(T_i \to B_i)$ where $T_i$ is at filtration level $f_{2i+1}$ and $B_i$ is at filtration level $f_{2i}$ for $i = 0, ... , k-1$.  Each of these bars is even.
\end{proof}

\begin{remark}
Jonathan Hanselman has informed the authors that a converse holds: If $\bo
(K) = 0$, then $K$ is an $L$-space knot.
\end{remark}

  Successful applications of Corollary~\ref{thm-evenodd} to low-crossing prime knots are rare.  In fact, the only prime knots of 12 or fewer crossings that have primality detected by Corollary~\ref{thm-evenodd}  are   $L$--space knots.  For other low-crossing knots,   the values of $\be(K)$ and $\bo(K)$ are close in value, both roughly $\delta(K) /4$;  Corollary~\ref{thm-evenodd} can be applied when one of the two is close to 0.

From our experiments, it seems that $\be(K) - \bo(K)$ is  related to the {\it dealternating} number of $K$; that is, the minimum number of crossing changes required to convert $K$ into an alternating knot. (References for the dealternating number include~\cite{MR2726563,MR2443110,MR4620858}.) These examples provide little evidence upon which to base a conjecture and the topic remains open for future study.

\appendix

\section{Multiplicativity of $\Omega_K(s,t)$.}
The existence of a bar-complex representative of $\CFKh(K)$ implies the simpler fact that $\CFKh(K)$ has a reduced representative, one for which the boundary map strictly lowers filtration levels.  Let $\widehat{C}(K)$ denote such a reduced complex with an ordered, graded, filtered  basis.  Then the coefficient $c_{i,j}$ of $s^jt^i$ in $\Omega_K(s,t)$ is precisely the number of basis elements at grading $j$ and at  filtration level $i$.

Let $\widehat{C}^1$ and $\widehat{C}^2$ be two such based reduced knot complexes.  Their associated polynomials have coefficients which we denote by  $c_{i,j}^\epsilon$ for $\epsilon = 1, 2$.    The tensor product of  reduced complexes is  clearly reduced.  Denote the tensor product by $\widehat{D}$  and denote the coefficients of  its associated polynomial  by $d_{i,j}$.  A simple counting argument shows that $d_{m,n} = \sum_{ x+ y  = m}  \big(  \sum_{z + w = n} c_{x,z}^1 c_{y,w}^2\big)$.  This is precisely what arises when the polynomials are multiplied.

%%%%%%%Bibliography%%%%%%%%%%%%%%

% This line increases the horizontal distance between the citation labels and the citations
%\setlength{\labelsep}{10em}

\bibliographystyle{amsplain}

\bibliography{../../../../BibTexComplete}

\end{document}